\documentclass{amsart}

\usepackage{amsmath,amscd,amssymb,amsthm}
\usepackage{amsfonts}
\usepackage{geometry}
\usepackage[
  colorlinks,
  linkcolor=blue,
  breaklinks=true,
  urlcolor=blue,
  citecolor=blue,
  anchorcolor=blue,
  pagebackref
]{hyperref}

\newtheorem{theorem}{Theorem}[section]
\newtheorem{lemma}{Lemma}[section]
\newtheorem{remark}{Remark}
\newtheorem*{example}{Example}
\newtheorem{conjecture}{Conjecture}

\newcommand{\R}{\mathbb R}
\newcommand{\C}{\mathbb C}
\newcommand{\ii}{\sqrt{-1}}
\newcommand{\CP}{\mathbb{CP}}
\newcommand{\Fm}{\mathcal F_m}
\newcommand{\Ric}{\operatorname{Ric}}
\newcommand{\SU}{\operatorname{SU}}
\newcommand{\SL}{\operatorname{SL}}
\newcommand{\Hom}{\operatorname{Hom}}
\newcommand{\mix}{\mathcal C_{\alpha,\beta}}
\newcommand{\dd}{\,d\sigma}
\newcommand{\doi}[1]{%
  \href{https://doi.org/#1}{doi:\nolinkurl{#1}}%
}

\renewcommand*\backref[1]{}
\renewcommand*\backrefalt[4]{%
  \ifcase #1%
  \or\space(cited on page #2)%
  \else\space(cited on pages #2)%
  \fi
}

\hypersetup{
  pdftitle={
    Mixed Chern curvature on complete flags
  },
  pdfauthor={Yuqin Guo and Dongmei Zhang}
}

\begin{document}

\title[Mixed Chern curvature on complete flags]
{Mixed Chern curvature on complete flags}

\author{Yuqin Guo}
\address{
Yuqin Guo.
School of Mathematical Sciences,
Chongqing Normal University,
Chongqing 401331, China
}
\email{1942747285@qq.com}

\author{Dongmei Zhang}
\address{Dongmei Zhang. School of Mathematical Sciences, Chongqing Normal University, Chongqing 401331, China}
\email{{2250825921@qq.com}}

%
%
%

\date{}



\begin{abstract}
We prove that the complete flag manifold $\SU(m)/T^{m-1}$, with $m\geq3$ and its standard complex structure, admits no invariant Hermitian metric with constant mixed Chern curvature when the coefficient of the holomorphic sectional curvature is nonzero. This extends the preceding nonexistence result for a particular Wallach metric on $\SU(3)/T^2$ to all invariant
Hermitian metrics and to complete type-A flags of arbitrary dimension. Curvature evaluations on holomorphic root spheres contradict a necessary trace identity, covering both K\"ahler and non-K\"ahler invariant metrics and excluding both zero and nonzero curvature constants.
\end{abstract}

\maketitle

\markleft{Yuqin Guo and Dongmei Zhang}
\markright{Mixed Chern curvature on complete flags}

\section{Introduction and statement of result}

Complex geometry harbours a long--standing open problem concerning classification of compact Hermitian manifolds endowed with constant holomorphic sectional curvature.
In the K\"ahler case, it is a well established fact that every compact (or complete) K\"ahler manifold of constant holomorphic sectional curvature belongs to the family of complex space forms.
In other words, their universal covers are one of the standard models: complex projective space$\,\mathbb{CP}^n$, complex Euclidean space $\mathbb{C}^n$, or complex hyperbolic space $\mathbb{CH}^n$.

For non-K\"ahler metrics $g$, the curvature tensor $R$ fails to satisfy the symmetry identities valid in the K\"ahler framework. Consequently, the holomorphic sectional curvature $H$ alone cannot recover full information encoded within $R$.
Nevertheless, the following statement remains widely conjectured to hold for compact manifolds.

\begin{conjecture}[
Constant Chern holomorphic sectional curvature;
original formulation
]\label{conj:hsc}
Let $(M^n,J,g)$ be a compact Hermitian manifold with $n\geq2$ and $H\equiv c$. If $c\neq0$, then $g$ is K\"ahler; if $c=0$, then $g$ is Chern flat, namely $R\equiv0$.
\end{conjecture}

Note that the compactness assumption is necessary for the validity of the conjecture, without it there are counterexamples in the noncompact case. Compact Chern flat manifolds have been classified by Boothby \cite{Boothby} as all the compact quotients of complex Lie groups equipped with left invariant metrics. Moreover, compact Levi-Civita flat Hermitian threefolds have been classified by Khan-Yang-Zheng in \cite{KYZ}, but for $n\geq4$, the classification remains an open question.

In complex dimension $n=2$ , combined contributions from Balas-Gauduchon \cite{Balas, BG} (1985), Sato-Sekigawa \cite{SS} (1990), and Apostolov-Davidov-Muskarov \cite{ADM} (1996) have confirmed Conjecture \ref{conj:hsc}.
For $n\ge 3$ , the first substantive result towards this conjecture originates from Davidov-Grantcharov-Muskarov \cite{DGM}. Their work shows that $\mathbb{CP}^3$ is the unique twistor space admitting constant holomorphic sectional curvature.
More recently, Chen-Chen-Nie \cite{CCN} verified Conjecture \ref{conj:hsc} for locally conformally K\"ahler manifolds satisfying $c\leq 0$.
Tang \cite{TangKL} validated the conjecture under an extra hypothesis that the metric is Chern-K\"ahler-like, meaning its Chern curvature $R^c$ obeys all K\"ahler symmetry rules.
Zhou-Zheng \cite{ZhouZ} obtained another confirmation for compact balanced threefolds with vanishing real bisectional curvature a condition introduced in \cite{YangZ}, which is marginally stronger than the condition encoded by $ H^c $. For Hermitian manifolds, Tang \cite{tang-2026} has given a complete solution to the problem of constant real bisectional curvature.
Their argument yields Chern-flatness for such manifolds. Further affirmative instances for $\nabla^c$ were obtained by Li-Zheng \cite{LZ} and Rao-Zheng \cite{RZ}, working under additional assumptions that either $(M^n,g)$ is a complex nilmanifold or g satisfies the Bismut-K\"ahler-like condition, so that the Bismut curvature $R^b$ satisfies all K\"ahler symmetries.

Recent preprints settle the unrestricted threefold case: Qin--Tian \cite{QT} established nonzero-constant rigidity, and
Chen--Zheng \cite{CZthree} established Chern flatness for $c=0$, removing the balanced assumption in Chen--Li \cite{CL}.
Qin--Tian's version~3 also proves K\"ahler rigidity for $c\leq0$ in Fujiki's class $\mathcal C$ and reports compact balanced metrics with $H_g\equiv0$ but $R\not\equiv0$ for every $n\geq7$. If confirmed, these reported counterexamples disprove the unrestricted zero-curvature assertion in those dimensions.

Geometers have also achieved abundant results on the mixed-curvature problem. Mixed curvature was introduced by Chu--Lee--Tam in the K\"ahler setting \cite{CLT}. Its special case $\alpha=0$ leads to the classical Chern holomorphic sectional curvature problem; see \cite{Zsurvey} for a survey. In the above notation, we have $\mathcal{C}_{1,0} = \operatorname{Ric}$ and $\mathcal{C}_{0,1} = H$. When the metric $g$ is K\"ahler, the mixed curvature encompasses several interesting curvature conditions.
Specifically, $\mathcal{C}_{1,1}$ corresponds to the $\operatorname{Ric}^+(X, \overline{X})$ introduced by Ni \cite{ni-2004}, and $\mathcal{C}_{1,-1}$ corresponds to the orthogonal Ricci curvature $\operatorname{Ric}^\perp(X, \overline{X})$ introduced by Ni-Zheng \cite{ni-zheng-2008}.
Additionally, if $k$ is an integer between $1$ and $n$, then $\mathcal{C}_{k-1,n-k}$ is closely related to the $k$-Ricci curvature $Ric_k$ introduced by Ni \cite{ni-2005}.

For compact K\"ahler manifolds with $\mathcal{C}_{\alpha,\beta} \geq 0$ or $\mathcal{C}_{\alpha,\beta} \leq 0$, there have been important progress made in recent years.
For example, Yang \cite{yang-2007} proved that a compact K\"ahler manifold with $\mathcal{C}_{0,1} > 0$ must be projective and rationally connected, confirming a conjecture by Yau \cite{yau-1978}.
Ni \cite{ni-2004} generalized this result, showing that the same holds if $Ric_k > 0$ for some $1 \leq k \leq n$.
Matsumura \cite{matsumura-1983}, \cite{matsumura-1985} established structure theorems for projective manifolds with $\mathcal{C}_{0,1} \geq 0$, while Chu-Lee-Tam \cite{CLT} and others proved that compact K\"ahler manifolds with $\mathcal{C}_{\alpha,\beta} > 0$, where $\alpha > 0$ and $3\alpha + 2\beta \geq 0$, must be projective and simply-connected.
Zhang-Zhang \cite{zhang-zhang-2010} extended Yang's result to the quasi-positive case, confirming a conjecture by Yang.
Tang \cite{tang-2012} demonstrated the projectivity of compact K\"ahler manifolds with quasi-positive $\mathcal{C}_{\alpha,\beta}$.
More recently, Chu-Lee-Zhu \cite{chu-lee-zhu-2018} established a structure theorem for compact K\"ahler manifolds with $\mathcal{C}_{\alpha,\beta} \geq 0$.
In the non-positive case, Wu-Yau \cite{wu-yau-2014} confirmed a conjecture of Yau that projective K\"ahler manifolds with $\mathcal{C}_{0,1} < 0$ must have ample canonical line bundles.
Tosatti-Yang \cite{tosatti-yang-2013} showed that the projectivity assumption can be removed in Wu-Yau's Theorem, and Chu-Lee-Tam \cite{CLT} proved that compact K\"ahler manifolds with $\mathcal{C}_{\alpha,\beta} < 0$ and $\alpha \geq 0$, $\beta \geq 0$, must have ample canonical bundles.

Most existing results for mixed-curvature invariants focus predominantly on the K\"ahler setting; comparatively little is understood for non-K\"ahler geometries. In \cite{Tang}, Tang formulated the following conjecture:

\begin{samepage}
\begin{conjecture}[Tang's conjecture]\label{conj:mixed}
Let $(M^n,J,g)$ be a compact Hermitian manifold with $\mix\equiv c$ and $\beta\neq0$. If $c\neq0$, then $g$ is K\"ahler.
\end{conjecture}
\end{samepage}

Chen--Zheng \cite{CZ} verified this conjecture for several Lie--Hermitian classes and compact balanced BTP threefolds.
Chen--Tang \cite{CT} studied other Chern--Ricci contractions, and Tang--Wang \cite{TW} proved nonzero-constant rigidity for compact LCK metrics with constant $k$th-mixed curvature. The unrestricted conjecture remains open; zero mixed curvature requires separate treatment. Our theorem excludes both zero and nonzero constants on complete flags independently of either conjecture.

Compact complex homogeneous spaces constitute a natural framework for studying the constant mixed-curvature condition. Their invariant Hermitian metrics are characterized via isotropy representations, and their first Chern class may be computed using homogeneous line bundles.

Even so, mere transitivity does not guarantee that the mixed curvature is independent of tangent directions. On a complete flag manifold, the distinct root directions yield a finite set of geometrically distinguished directions, against which one may verify the constant-curvature condition. We consider the complete type-A flag manifold
\begin{equation}\label{eq:flag-intro}
 \Fm=\SU(m)/T^{m-1},
 \qquad
 n=\dim_{\C}\Fm=\frac{m(m-1)}2,
\end{equation}
where $T^{m-1}$ is the diagonal maximal torus. We equip $\Fm$ with the standard integrable complex structure obtained from the space of complete flags in $\C^m$. Our main result is the following.

\begin{theorem}\label{thm:main}
Let $m\geq3$, and let $g$ be an $\SU(m)$-invariant Hermitian metric on $\Fm$ with its standard invariant complex structure. There are no real numbers $\alpha,\beta,c$ with $\beta\neq0$ such that
\[
 \mix(X)=c
 \qquad
 \text{for every nonzero }X\in T^{1,0}\Fm.
\]
\end{theorem}

The theorem concerns all invariant Hermitian metrics, rather than a single normal or canonical metric. In particular, it includes both K\"ahler and non-K\"ahler metrics, and excludes constant mixed curvature even when the proposed constant is zero. The restriction $m\geq3$ is sharp: $\mathcal F_2=\CP^1$ has constant mixed curvature for every invariant Hermitian metric. The assumption $\beta\neq0$ is also essential, since invariant K\"ahler--Einstein metrics on complete flag manifolds have constant first Ricci curvature.

There is an important preceding special case. In the proof of their Theorem~3, Chen--Zheng \cite[Section~6]{CZ} show that the particular non-K\"ahler Wallach metric on the threefold
\[
 \SU(3)/T^2\simeq\mathbb P(T\CP^2)
\]
cannot have constant mixed curvature. Theorem~\ref{thm:main} extends this nonexistence statement to every invariant Hermitian metric on that threefold and to complete type-A flags of arbitrary dimension. No parallel-torsion assumption is used. Our use of invariant flag geometry is compatible with the framework of Grama--Oliveira \cite{GO}; their scalar-curvature condition differs from the mixed-curvature condition studied here.

The proof requires only selected directional values of $R$, rather than a formula for all of its components. Every root line is tangent to a holomorphically embedded sphere. Since twice a root is not a root, the isotropy representation forces the Chern second fundamental form of this sphere to vanish. The resulting holomorphic sectional curvature and the anticanonical degree yield, respectively,
\[
 H(e_{ij})=\frac2{x_{ij}},
 \qquad
 \Ric^{(1)}(e_{ij},\bar e_{ij})
 =\frac{2(j-i)}{x_{ij}},
\]
where $e_{ij}$ is a unit root vector and $x_{ij}>0$ is the scale of the induced sphere metric. A spherical average then eliminates the parameters $\alpha$ and $c$ and leaves a sum of nonnegative terms with at least one strictly positive term.

Section~2 establishes the averaging identity and the necessary directional trace condition. Section~3 describes invariant metrics, balancedness, and holomorphic root spheres. Section~4 computes their anticanonical degrees and first Chern Ricci curvature. Section~5 proves Theorem~\ref{thm:main}, gives the explicit three-dimensional calculation, and explains the boundary cases.

\section{Averaging and the directional trace identity}
Let $(M^n,J,g)$ be a Hermitian manifold, and let $R$ denote the curvature tensor of its Chern connection. For a nonzero vector
$X\in T^{1,0}M$, write
\begin{equation}\label{eq:definition}
 H(X)
 =
 \frac{R(X,\bar X,X,\bar X)}{|X|^4},
 \qquad
 \mix(X)
 =
 \alpha
 \frac{\Ric^{(1)}(X,\bar X)}{|X|^2}
 +\beta H(X),
\end{equation}
where $\Ric^{(1)}$ is the first Chern Ricci tensor and $\alpha,\beta\in\R$. Throughout the paper, $\beta\neq0$. We use the term mixed curvature for this real linear combination without imposing a sign condition on either coefficient. The notation $\mix\equiv c$ means that the same real number $c$ is obtained at every point and in every nonzero $(1,0)$ tangent direction. In holomorphic coordinates, write the fundamental form as
\[
 \omega
 =
 \ii\sum_{i,j}g_{i\bar j}\,dz^i\wedge d\bar z^j.
\]
Our Chern curvature convention is
\begin{equation}\label{eq:curvature}
 R_{i\bar j k\bar\ell}
 =
 -\partial_i\partial_{\bar j}g_{k\bar\ell}
 +
 \sum_{r,t}g^{r\bar t}
 (\partial_i g_{k\bar t})
 (\partial_{\bar j}g_{r\bar\ell}).
\end{equation}
Here $(g^{r\bar t})$ denotes the inverse Hermitian tensor, so that
\[
 \sum_tg^{r\bar t}g_{k\bar t}=\delta^r_k.
\]
With this normalization, the metric
\[
 g_{w\bar w}=x(1+|w|^2)^{-2}
\]
on $\CP^1$ has holomorphic sectional curvature $2/x$.

Under a unitary frame $e_1,\ldots,e_n$, put
\begin{equation}\label{eq:traces}
 P_{i\bar j}
 =
 \Ric^{(1)}_{i\bar j}
 =
 \sum_kR_{i\bar j k\bar k},
 \qquad
 s=\sum_iP_{i\bar i},
 \qquad
 \widetilde s
 =
 \sum_{i,k}R_{i\bar k k\bar i}.
\end{equation}
Both $s$ and $\widetilde s$ are real. We do not identify these two contractions on an arbitrary Hermitian manifold. The first Chern Ricci form is
\begin{equation}\label{eq:ricci-form}
 \rho^{(1)}
 =
 \ii\sum_{i,j}P_{i\bar j}\,dz^i\wedge d\bar z^j
 =
 -\ii\partial\bar\partial\log\det(g_{i\bar j}),
 \qquad
 [\rho^{(1)}]=2\pi c_1(M).
\end{equation}
It is closed without a K\"ahler assumption.

For the full Chern torsion components, we use
\[
 T^k_{ij}
 =
 \sum_\ell g^{k\bar\ell}
 \bigl(
   \partial_i g_{j\bar\ell}
   -
   \partial_j g_{i\bar\ell}
 \bigr),
 \qquad
 \eta_i=\sum_kT^k_{ik}.
\]
The metric is balanced if $d\omega^{n-1}=0$, equivalently if the torsion trace form $\eta$ vanishes. With the full torsion convention above, the Chern Bianchi identity in a unitary frame reads
\[
 T^k_{ij,\bar\ell}
 =
 R_{j\bar\ell i\bar k}
 -
 R_{i\bar\ell j\bar k};
\]
see \cite{YZ}. Here a comma denotes Chern covariant differentiation. Contracting this identity gives
\begin{equation}\label{eq:torsion-trace}
 \sum_i\eta_{i,\bar i}
 =
 \sum_{i,k}
 \bigl(
   R_{k\bar i i\bar k}
   -
   R_{i\bar i k\bar k}
 \bigr)
 =
 \widetilde s-s.
\end{equation}
Consequently, every balanced Hermitian metric satisfies $s=\widetilde s$ pointwise. This conclusion requires neither compactness nor homogeneity.

We include the standard averaging argument below for the sake of clarity \cite{CT}.
\begin{lemma}\label{lem:average}
If a Hermitian manifold of complex dimension $n$ satisfies $\mathcal{C}_{\alpha,\beta}\equiv c$, then at every point
\begin{equation}\label{eq:average}
 n(n+1)c
 =
 \bigl((n+1)\alpha+\beta\bigr)s
 +
 \beta\widetilde s.
\end{equation}
\end{lemma}

\begin{proof}
Choose a unitary frame at the point and write
\[
 X=\sum_rz_re_r,
 \qquad
 \sum_r|z_r|^2=1.
\]
Let $d\sigma$ be probability measure on the unit sphere in $\C^n$. Unitary invariance gives
\begin{equation}\label{eq:moments}
 \int z_i\bar z_j\dd
 =
 \frac{\delta_{ij}}n,
 \qquad
 \int z_i\bar z_jz_k\bar z_\ell\dd
 =
 \frac{
   \delta_{ij}\delta_{k\ell}
   +
   \delta_{i\ell}\delta_{kj}
 }{n(n+1)}.
\end{equation}

To make the exposition transparent, one may compute the fourth-moment constant explicitly. The invariance property with respect to diagonal unitary matrices forces the integral to vanish whenever the multisets formed by holomorphic and antiholomorphic indices do not coincide. For $n\geq2$, define
\[
 A=\int|z_1|^4\dd,
 \qquad
 B=\int|z_1z_2|^2\dd.
\]
The random variable $(z_1+z_2)/\sqrt2$ is identically distributed to $z_1$. Expanding its fourth power and invoking phase invariance yields $A=A/2+B$, so $A=2B$. Furthermore,
\[
 1
 =
 \int
 \left(\sum_r|z_r|^2\right)^2
 \dd
 =
 nA+n(n-1)B.
\]
Thus
\[
 B=\frac1{n(n+1)},
 \qquad
 A=\frac2{n(n+1)},
\]
proving \eqref{eq:moments}. Its formula also holds for $n=1$.

The second moment yields
\[
 \int P(X,\bar X)\dd=\frac{s}{n}.
\]
The two fourth-moment contractions yield
\[
 \int R(X,\bar X,X,\bar X)\dd
 =
 \frac{s+\widetilde s}{n(n+1)}.
\]
Since $X$ is unit and $\mix(X)=c$, averaging
\eqref{eq:definition} and multiplying by $n(n+1)$ proves
\eqref{eq:average}.
\end{proof}

\begin{lemma}\label{lem:directional}
Let $(M^n,J,g)$ be a balanced Hermitian manifold. If $\mix\equiv c$ with $\beta\neq0$, then every unitary basis at every
point satisfies
\begin{equation}\label{eq:directional}
 \sum_{r=1}^nH(e_r)
 =
 \frac{2s}{n+1}.
\end{equation}
\end{lemma}

\begin{proof}
Balancedness and \eqref{eq:torsion-trace} give $s=\widetilde s$ at the chosen point. Equation~\eqref{eq:average} becomes
\begin{equation}\label{eq:averaged-c}
 nc
 =
 \alpha s+\frac{2\beta s}{n+1}.
\end{equation}
On the other hand, evaluating the mixed-curvature condition at each member of any chosen unitary frame and summing gives
\begin{equation}\label{eq:summed-c}
 \begin{aligned}
 nc
 &=
 \alpha\sum_rP(e_r,\bar e_r)
 +
 \beta\sum_rH(e_r)\\
 &=
 \alpha s+\beta\sum_rH(e_r).
 \end{aligned}
\end{equation}
Subtracting \eqref{eq:averaged-c} from \eqref{eq:summed-c} and using $\beta\neq0$ proves \eqref{eq:directional}.
\end{proof}

This cancellation proceeds solely via division by $\beta\neq0$, and never involves division by $c$, $\alpha$, or $(n+1)\alpha+2\beta$. Consequently, all choices for the sign of the curvature constant and all coefficient ratios are encompassed by this cancellation. For an arbitrary Hermitian metric, the sum appearing on the left-hand side of Equation \eqref{eq:directional} is not necessarily invariant under frame changes; the hypothesis of constant mixed curvature imposes that this expression attains the stated value in any unitary frame.

\section{Invariant geometry and holomorphic root spheres}
In this section, we turn to the geometry of the complete type-A flag manifold $\mathcal{F}_m=\mathrm{SU}(m)/T^{m-1}$. We first describe its isotropy representation and parameterize all $\mathrm{SU}(m)$-invariant Hermitian metrics on $\mathcal{F}_m$, without imposing the K\"ahler condition. The root decomposition of the tangent space yields a family of holomorphic root spheres embedded in $\mathcal{F}_m$. These holomorphic root spheres are the key submanifolds where we carry out explicit curvature computations, which will be combined with the trace identity derived in Section 2 to establish our nonexistence theorem. Realize $\Fm$ as the set of complete flags
\[
 0=E_0\subset E_1\subset\cdots\subset E_m=\C^m,
 \qquad
 \dim_\C E_k=k.
\]
Let $\epsilon_1,\ldots,\epsilon_m$ be the standard basis of $\C^m$, and let $o$ denote the standard flag
\[
 E_k^0
 =
 \operatorname{span}_{\C}
 \{\epsilon_1,\ldots,\epsilon_k\}.
\]
Its stabilizer in $\SU(m)$ is $T^{m-1}$. For $i<j$, let $v_{ij}$ be the tangent vector represented by the matrix unit $E_{ji}$ sending $\epsilon_i$ to $\epsilon_j$, and put
\[
 V_{ij}=\C v_{ij}.
\]
Then
\begin{equation}\label{eq:root-decomposition}
 T_o^{1,0}\Fm
 =
 \bigoplus_{1\leq i<j\leq m}V_{ij}.
\end{equation}
An element $\operatorname{diag}(t_1,\ldots,t_m)$ of the torus acts on $V_{ij}$ by the character
\begin{equation}\label{eq:character}
 \chi_{ij}(t)=t_jt_i^{-1}.
\end{equation}
The characters are distinct. Indeed, torus invariance gives
\[
 g(v_{ij},\bar v_{ab})
 =
 \chi_{ij}(t)
 \overline{\chi_{ab}(t)}
 g(v_{ij},\bar v_{ab}),
 \qquad
 t\in T^{m-1},
\]
so different root lines are orthogonal. Define
\begin{equation}\label{eq:root-normalization}
 x_{ij}=g(v_{ij},\bar v_{ij})>0,
 \qquad
 e_{ij}=x_{ij}^{-1/2}v_{ij}.
\end{equation}
Then $\{e_{ij}:i<j\}$ is a unitary basis at $o$. These positive parameters describe all invariant Hermitian metrics; no condition such as $d\omega=0$ is imposed. Only a basis at a point is used, and no global holomorphic unitary frame is required.

\begin{lemma}\label{lem:balanced}
Every $\SU(m)$-invariant Hermitian metric on $\Fm$ is balanced.
\end{lemma}

\begin{proof}
Set $\mathcal{F}_m=\mathrm{SU}(m)/T^{m-1}$. Let $o=eT^{m-1}$ denote the standard flag, and take $g$ to be an $\SU(m)$-invariant Hermitian metric whose associated fundamental form is $\omega$. First, $\delta_g\omega$ is an invariant real one-form.  Indeed, for every $a\in\SU(m)$, the left translation $L_a$ preserves $g$, the complex orientation, and $\omega$.  An orientation-preserving isometry commutes with the Hodge star; hence it also commutes with the codifferential. Therefore
\[
 L_a^*(\delta_g\omega)
   =\delta_g(L_a^*\omega)
   =\delta_g\omega.
\]
Since the isotropy group at $o$ is $T^{m-1}$, it follows that
\[
 (\delta_g\omega)_o\in (T_o^*\Fm)^{T^{m-1}}.
\]
We next prove that this fixed subspace is zero.  Recall that
\[
 T_o^{1,0}\Fm=\bigoplus_{1\leq i<j\leq m}V_{ij},
 \qquad V_{ij}=\C v_{ij},
\]
and that $t=\operatorname{diag}(t_1,\ldots,t_m)\in T^{m-1}$ acts on
$V_{ij}$ through the character
\[
 \chi_{ij}(t)=t_jt_i^{-1}.
\]
Consequently, the complexified real tangent representation and its dual have the weight decompositions
\[
 T_o\Fm\otimes_{\R}\C
 =\bigoplus_{i<j}
   \left(\C_{\chi_{ij}}\oplus\C_{\chi_{ij}^{-1}}\right),
\]
and
\[
 T_o^*\Fm\otimes_{\R}\C
 =\bigoplus_{i<j}
   \left(\C_{\chi_{ij}^{-1}}\oplus\C_{\chi_{ij}}\right),
\]
where $\C_\chi$ denotes the one-dimensional representation with character $\chi$. Passing to the dual only replaces each character by its inverse.

Every character occurring in the last decomposition is nontrivial.  For a fixed pair $i<j$, consider the one-parameter subgroup of $T^{m-1}$ defined by
\[
 t_i=e^{-\ii\theta},\qquad
 t_j=e^{\ii\theta},\qquad
 t_k=1\quad(k\neq i,j).
\]
Its determinant equals one, and
\[
 \chi_{ij}(t)=e^{2\ii\theta}.
\]
Thus neither $\chi_{ij}$ nor $\chi_{ij}^{-1}$ is the trivial character. Since a torus representation is the direct sum of its weight spaces, a fixed vector can have a component only in the zero-weight space.  There is no such space here, and therefore
\[
 \left(T_o^*\Fm\otimes_{\R}\C\right)^{T^{m-1}}=\{0\}.
\]
If a nonzero real covector were fixed by $T^{m-1}$, its complexification would be a nonzero fixed vector in the complexified representation.  Hence
\[
 (T_o^*\Fm)^{T^{m-1}}=\{0\}.
\]
It follows that $(\delta_g\omega)_o=0$.  The $\SU(m)$-invariance of $\delta_g\omega$ and the transitivity of the group action then give
\[
 \delta_g\omega=0
\]
on all of $\Fm$.

Finally, we relate this equation to balancedness.  On a Hermitian manifold of complex dimension $n$, endowed with the complex orientation, the pointwise Hermitian linear-algebra identity is
\[
 *_g\omega=\frac{\omega^{n-1}}{(n-1)!}.
\]
For a real two-form in real dimension $2n$, the codifferential is
$\delta_g=-*_g d *_g$.  Consequently,
\[
 \delta_g\omega
   =-\frac{1}{(n-1)!}*_g d\omega^{n-1}.
\]
The Hodge star is invertible, so
\[
 \delta_g\omega=0
 \quad\Longleftrightarrow\quad
 d\omega^{n-1}=0.
\]
Thus every $\SU(m)$-invariant Hermitian metric on $\Fm$ is balanced, as claimed in Lemma~3.1.
\end{proof}
For $i<j$, let $S_{ij}\simeq\SU(2)$ be the subgroup acting on $\operatorname{span}_{\C}\{\epsilon_i,\epsilon_j\}$ and fixing its orthogonal complement. Its orbit through $o$ will be denoted by $C_{ij}$. It is a holomorphic copy of $\CP^1$, as can be seen directly. For a line $L$ in that two-dimensional space, set
\begin{equation}\label{eq:sphere-flags}
 E_k(L)
 =
 \begin{cases}
 E_k^0,
 &k<i\text{ or }k\geq j,\\[2mm]
 \operatorname{span}_{\C}
 \{\epsilon_a:1\leq a\leq k,\ a\neq i\}
 \oplus L,
 &i\leq k<j.
 \end{cases}
\end{equation}
These subspaces are nested and have dimension $k$. For $i\leq k<j$, the fixed summand has dimension $k-1$ and is disjoint from $L$. On the chart $L(w) = \operatorname{span}_{\C} (\epsilon_i+w\epsilon_j)$, the subspaces depend holomorphically on $w$, and the same holds in the chart at infinity. Projection of $E_i(L)$ to the fixed plane $\operatorname{span}_{\C} \{\epsilon_i,\epsilon_j\}$ recovers $L$ and defines a holomorphic inverse on the image. Thus the map is a holomorphic embedding. Its image is $S_{ij}\cdot o$, with stabilizer $U(1)$, and its derivative at $w=0$ is $v_{ij}$.

The induced metric is $S_{ij}$-invariant. Since the tangent isotropy representation of $\SU(2)/U(1)$ is a complex line, it is a positive multiple of the Fubini--Study metric. Its value at $w=0$ is fixed by \eqref{eq:root-normalization}, giving
\begin{equation}\label{eq:sphere-metric}
 \omega|_{C_{ij}}
 =
 \frac{\ii x_{ij}}{(1+|w|^2)^2}
 \,dw\wedge d\bar w.
\end{equation}
In particular,
\begin{equation}\label{eq:area}
 \int_{C_{ij}}\omega
 =
 4\pi x_{ij}
 \int_0^\infty
 \frac{r\,dr}{(1+r^2)^2}
 =
 2\pi x_{ij}.
\end{equation}
Thus the sphere-area parameter and the squared norm of the root vector have the same normalization.

\begin{lemma}\label{lem:second-fundamental}
For every invariant Hermitian metric on $\Fm$, each $C_{ij}$ has vanishing Chern second fundamental form.
\end{lemma}

\begin{proof}
Along a holomorphic submanifold, $T^{1,0}C_{ij}$ is a holomorphic subbundle of $T^{1,0}\Fm|_{C_{ij}}$. The Chern connection has $(0,1)$ part equal to the holomorphic structure operator. Its $(0,1)$ covariant derivative therefore has zero normal projection. The remaining second fundamental form is
\[
 \mathrm{II}(X,Y)
 =
 (\nabla_XY)^\perp,
 \qquad
 X,Y\in T^{1,0}C_{ij},
\]
Here the orthogonal projection is defined with respect to the ambient metric. This tensor is complex bilinear. Upon normal projection, the derivative of the coefficient multiplying $Y$ yields a tangent vector.

Holomorphic isometries preserve both the Chern connection and the orthogonal normal projection. The torus fixes $o$ and keeps $C_{ij}$ invariant. It preserves the defining plane and all fixed summands in \eqref{eq:sphere-flags}. Thus the tensor evaluated at $o$ defines a torus-equivariant map.
\begin{equation}\label{eq:ii-map}
 V_{ij}\otimes V_{ij}
 \longrightarrow
 \bigoplus_{\substack{(a,b)\neq(i,j)\\a<b}}V_{ab}.
\end{equation}
Write
\[
 \mathrm{II}_o(v_{ij},v_{ij})
 =
 \sum_{\substack{(a,b)\neq(i,j)\\a<b}}
 A_{ab}v_{ab}.
\]
Equivariance implies
\[
 \chi_{ij}(t)^2A_{ab}
 =
 \chi_{ab}(t)A_{ab},
 \qquad
 t\in T^{m-1}.
\]
The exponent vector of $\chi_{ij}^2$ has entries $-2$ at $i$ and $2$ at $j$, while that of every root character has entries $-1$ and $1$. Both exponent vectors have sum zero. Equality of their characters on the determinant-one torus would make their difference a multiple of $(1,\ldots,1)$; the zero sum forces that multiple to vanish. The exponent vectors cannot agree, so every $A_{ab}$ is zero. Therefore \eqref{eq:ii-map} vanishes.

Finally, $S_{ij}$ acts transitively on $C_{ij}$ by ambient holomorphic isometries. The vanishing at $o$ transports to every point of the sphere.
\end{proof}

\begin{lemma}\label{lem:hsc}
If $e_{ij}$ is a unit vector in $V_{ij}$, then
\begin{equation}\label{eq:root-hsc}
 H(e_{ij})
 =
 \frac2{x_{ij}}.
\end{equation}
\end{lemma}

\begin{proof}
Set $E=T^{1,0}\Fm|_{C_{ij}}, S=T^{1,0}C_{ij}$. The orthogonal projection of the ambient Chern connection onto $S$ is metric-compatible and has $(0,1)$ part $\bar\partial_S$. It is therefore the Chern connection of the induced metric. With the convention \eqref{eq:curvature}, the subbundle Gauss identity gives
\begin{equation}\label{eq:chern-gauss}
 R^E(X,\bar X,Y,\bar Y)
 =
 R^S(X,\bar X,Y,\bar Y)
 +
 \|\mathrm{II}(X,Y)\|_g^2
\end{equation}
for tangent vectors $X,Y$. Lemma~\ref{lem:second-fundamental} makes the last term zero, so the ambient and intrinsic holomorphic
sectional curvatures agree in the tangent direction. This is a Hermitian holomorphic-subbundle identity and does not require a
K\"ahler metric.Set
\[
 h(w)
 =
 x_{ij}(1+|w|^2)^{-2}.
\]
At $w=0$,
\[
 h(0)=x_{ij},
 \qquad
 \partial_wh(0)=0,
 \qquad
 \partial_w\partial_{\bar w}h(0)
 =
 -2x_{ij}.
\]
Formula~\eqref{eq:curvature} yields
\[
 \begin{aligned}
 R_{w\bar w w\bar w}(0)
 &=
 -h_{w\bar w}(0)
 +
 h(0)^{-1}|h_w(0)|^2\\
 &=
 2x_{ij}.
 \end{aligned}
\]
Dividing by $h(0)^2=x_{ij}^2$ gives
\eqref{eq:root-hsc}.
\end{proof}

\section{Anticanonical degrees and the first Chern Ricci curvature}

The holomorphic sectional curvature in Lemma~\ref{lem:hsc} is determined by the induced sphere metric. The first Ricci contraction involves all tangent directions of the ambient flag manifold. Its value along the same sphere can instead be recovered from the first Chern class and invariance.

Let $\mathcal E_k$ be the rank-$k$ tautological holomorphic bundle on $\Fm$, and let $\mathcal L_k = \mathcal E_k/\mathcal E_{k-1}, 1\leq k\leq m,$
be the successive quotient line bundles. The determinant of the tangent bundle is
\begin{equation}\label{eq:det-tangent}
 \det T^{1,0}\Fm
 \simeq
 \bigotimes_{a<b}
 \Hom(\mathcal L_a,\mathcal L_b)
 \simeq
 \bigotimes_{k=1}^m
 \mathcal L_k^{\otimes(2k-m-1)}.
\end{equation}
Negative powers denote powers of the dual line bundle. To justify the isomorphisms, use
\[
 \Fm=\SL(m,\C)/B,
\]
where $B$ is the upper-triangular Borel subgroup. The holomorphic tangent bundle is associated to the $B$-representation
$\mathfrak{sl}(m,\C)/\mathfrak b,$ whose diagonal weights are $t_bt_a^{-1}$ for $a<b$. The unipotent radical acts unipotently on this representation and therefore has determinant one. Its determinant character is consequently
\[
 \prod_{a<b}\frac{t_b}{t_a}
 =
 \prod_{k=1}^m
 t_k^{(k-1)-(m-k)}
 =
 \prod_{k=1}^m
 t_k^{2k-m-1}.
\]
The fiber of $\mathcal L_k$ at the standard flag has character $t_k$, so the tensor products in \eqref{eq:det-tangent} have exactly this $B$-character. This proves an isomorphism of determinant line bundles, without asserting a holomorphic splitting of the tangent bundle into root line bundles. Using
\[
 c_1(T)=c_1(\det T),
 \qquad
 c_1(L\otimes Q)=c_1(L)+c_1(Q),
 \qquad
 c_1(L^*)=-c_1(L),
\]
gives
\begin{equation}\label{eq:c1-tangent}
 \begin{aligned}
 c_1(\Fm)
 &=
 \sum_{a<b}
 \bigl(
   c_1(\mathcal L_b)
   -
   c_1(\mathcal L_a)
 \bigr)\\
 &=
 \sum_{k=1}^m
 (2k-m-1)c_1(\mathcal L_k).
 \end{aligned}
\end{equation}

\begin{lemma}\label{lem:degree}
For every $i<j$, the anticanonical degree along $C_{ij}$ is
\begin{equation}\label{eq:degree}
 \int_{C_{ij}}c_1(\Fm)
 =
 2(j-i).
\end{equation}
\end{lemma}

\begin{proof}
Put $W=\operatorname{span}_{\C} \{\epsilon_i,\epsilon_j\}$ and regard $C_{ij}$ as the projective line of lines in $W$. The tautological exact sequence is
\[
 0
 \longrightarrow
 \mathcal O(-1)
 \longrightarrow
 \underline W
 \longrightarrow
 Q
 \longrightarrow
 0.
\]
Here $\underline W$ is the trivial rank-two bundle. The line bundle $Q$ has degree $1$, and hence is isomorphic to $\mathcal O(1)$. Construction~\eqref{eq:sphere-flags} identifies
\[
 \mathcal L_i|_{C_{ij}}
 \simeq
 \mathcal O(-1),
 \qquad
 \mathcal L_j|_{C_{ij}}
 \simeq
 Q
 \simeq
 \mathcal O(1).
\]
For $i<k<j$, the class of the fixed vector $\epsilon_k$ gives a nowhere-vanishing section of $\mathcal L_k|_{C_{ij}}$. For $k<i$ or $k>j$, the successive quotient is fixed. All these other line bundles are therefore trivial, and
\[
 \deg(\mathcal L_i|_{C_{ij}})=-1,
 \qquad
 \deg(\mathcal L_j|_{C_{ij}})=1,
 \qquad
 \deg(\mathcal L_k|_{C_{ij}})=0
 \quad(k\neq i,j).
\]
Substituting into \eqref{eq:c1-tangent} yields
\[
 -(2i-m-1)+(2j-m-1)
 =
 2(j-i).
 \qedhere
\]
\end{proof}

\begin{lemma}\label{lem:root-ricci}
For an invariant Hermitian metric on $\Fm$ and a unit vector
$e_{ij}\in V_{ij}$,
\begin{equation}\label{eq:root-ricci}
 P(e_{ij},\bar e_{ij})
 =
 \frac{2(j-i)}{x_{ij}}.
\end{equation}
Consequently,
\begin{equation}\label{eq:flag-scalar}
 s
 =
 \sum_{i<j}
 \frac{2(j-i)}{x_{ij}}.
\end{equation}
\end{lemma}

\begin{proof}
The ambient first Chern Ricci form $\rho^{(1)}$ is invariant under $\SU(m)$ because the metric and its Chern connection are invariant. Since $C_{ij}$ has real dimension two and $\omega|_{C_{ij}}$ is nowhere zero, write
\[
 \rho^{(1)}|_{C_{ij}}
 =
 f_{ij}\omega|_{C_{ij}}.
\]
Both forms are $S_{ij}$-invariant, so $f_{ij}$ is invariant. Transitivity on the connected sphere makes it a real constant
$r_{ij}$. Thus invariance, together with the degree calculation, determines the pointwise ratio. Equations~\eqref{eq:ricci-form}, \eqref{eq:area}, and \eqref{eq:degree} give
\[
 \begin{aligned}
 r_{ij}\,2\pi x_{ij}
 &=
 \int_{C_{ij}}\rho^{(1)}\\
 &=
 2\pi\int_{C_{ij}}c_1(\Fm)\\
 &=
 4\pi(j-i).
 \end{aligned}
\]
Hence
\[
 r_{ij}
 =
 \frac{2(j-i)}{x_{ij}}.
\]
For a unit tangent vector,
\[
 \omega(e_{ij},\bar e_{ij})=\ii,
 \qquad
 \rho^{(1)}(e_{ij},\bar e_{ij})
 =
 \ii P(e_{ij},\bar e_{ij}),
\]
which proves \eqref{eq:root-ricci}. The orthogonal root lines form a unitary basis at $o$, so taking the trace gives \eqref{eq:flag-scalar}. Invariance gives the same scalar at every point.
\end{proof}

The form restricted in this argument is the first Chern Ricci form of the ambient tangent bundle, not the Ricci form of the sphere's own tangent bundle. This distinction accounts for the factor $j-i$: the intrinsic tangent bundle of $\CP^1$ has degree two, whereas the restricted ambient anticanonical bundle has degree $2(j-i)$.

\section{Proof of the main theorem and boundary cases}

\begin{proof}[Proof of Theorem~\ref{thm:main}]
Suppose that an invariant Hermitian metric satisfies $\mix\equiv c$ with $\beta\neq0$. The metric is balanced by Lemma~\ref{lem:balanced}, so Lemma~\ref{lem:directional} applies to the unitary root basis $\{e_{ij}:i<j\}$. Substitution of \eqref{eq:root-hsc} and \eqref{eq:flag-scalar} gives
\[
 \sum_{i<j}\frac2{x_{ij}}
 =
 \frac2{n+1}
 \sum_{i<j}
 \frac{2(j-i)}{x_{ij}}.
\]
Multiplying by $(n+1)/2$ and moving all terms to one side, we obtain
\begin{equation}\label{eq:contradiction}
 \sum_{i<j}
 \frac{n+1-2(j-i)}{x_{ij}}
 =
 0.
\end{equation}
For each root, $j-i\leq m-1$. Therefore
\begin{align}\label{eq:nonnegative}
 n+1-2(j-i)
 &\geq
 \frac{m(m-1)}2+1-2(m-1)
 \notag\\
 &=
 \frac{(m-2)(m-3)}2
 \geq0.
\end{align}
Every denominator in \eqref{eq:contradiction} is strictly positive. For any simple root, $j=i+1$, its numerator is $n+1-2=n-1>0,$ since $m\geq3$ gives $n\geq3$. Thus the left-hand side of \eqref{eq:contradiction} is strictly positive, contradicting its vanishing. This proves the theorem.
\end{proof}
The proof finds an obstruction with no dependence on parameters. Comparing the spherical average against the sum over root directions eliminates $\alpha$ and $c$. No special ratio between $\alpha$ and $\beta$ requires inspection. When $m=3$, the highest-root coefficient is zero. The simple-root coefficients remain positive. When $m\geq4$, every coefficient in \eqref{eq:contradiction} is strictly positive.

\begin{example}\label{ex:threefold}
On $\mathcal F_3=\SU(3)/T^2$, set
\[
 x=x_{12},
 \qquad
 y=x_{23},
 \qquad
 z=x_{13}.
\]
The directional values are
\[
 (H_{12},H_{23},H_{13})
 =
 (2/x,2/y,2/z),
\]
and
\[
 (P_{12},P_{23},P_{13})
 =
 (2/x,2/y,4/z).
\]
Since $n=3$, the necessary trace identity would be
\[
 \frac2x+\frac2y+\frac2z
 =
 \frac12
 \left(
   \frac2x+\frac2y+\frac4z
 \right).
\]
It reduces to
\[
 \frac1x+\frac1y=0,
\]
which is impossible. Thus the threefold obstruction holds for every positive triple $(x,y,z)$; it does not depend on the scale in the nonsimple root direction. In particular, the previously studied Wallach metric is one member of the excluded family.
\end{example}

\begin{remark}[The K\"ahler subcone and the coefficient restriction]
In the normalization above, invariant K\"ahler metrics satisfy the additive relations
\[
 x_{ik}=x_{ij}+x_{jk},
 \qquad
 i<j<k.
\]
The proof never uses these relations, so it treats the full Hermitian cone. To see concretely why $\beta\neq0$ is needed, choose
\[
 x_{ij}
 =
 \frac{2(j-i)}\lambda,
 \qquad
 \lambda>0.
\]
These parameters satisfy the additive relations and define an invariant K\"ahler metric. Formula~\eqref{eq:root-ricci} gives
\[
 P(e_{ij},\bar e_{ij})=\lambda
\]
in every root direction. Both $P$ and $g$ are torus-invariant Hermitian tensors and are diagonal in the root decomposition, so
$P=\lambda g$. Hence
\[
 \mathcal C_{\alpha,0}
 \equiv
 \alpha\lambda
\]
for this metric.
\end{remark}

\begin{remark}[The one-dimensional case]
For $m=2$, the flag manifold is $\CP^1$ and there is only one scale $x_{12}=x$. Its first Ricci curvature equals its holomorphic
sectional curvature, both being $2/x$. Therefore
\[
 \mix
 \equiv
 \frac{2(\alpha+\beta)}x.
\]
This verifies the sharpness of the dimension restriction, including the case of a zero mixed-curvature constant.
\end{remark}

The choice $\alpha=0$ in Theorem~\ref{thm:main} also excludes constant Chern holomorphic sectional curvature on these invariant
metrics. The theorem concerns the Chern connection and the standard invariant complex structure. It makes no assertion about non-invariant Hermitian metrics on the same manifold. Its extension to another homogeneous space would require the corresponding tangent-direction curvature evaluations and degree bounds to be established for that space.\\

\noindent\textbf{Acknowledgments.} We would like to express our sincere gratitude to Professor Fangyang Zheng for his insightful guidance, valuable discussions, and continued support throughout the course of this research.

\end{document}